\documentclass[12pt]{amsart}
\usepackage{amsmath}
\usepackage{amsthm}
\usepackage{amssymb}
\usepackage{amsmath}
\usepackage{amsthm}
\usepackage{amssymb}
\usepackage{fancyhdr}
\usepackage{enumerate}
\usepackage{amscd}
\usepackage[all]{xy}
\usepackage{graphicx}
\usepackage{color}
\usepackage{soul}%For the command \st{text} to cross out "text"

\theoremstyle{remark}

\newtheorem*{pf}{Proof}
\newtheorem*{pfam}{Proof of Theorem \ref{thm:am}}

\newtheorem*{pfctm}{Proof of Theorem \ref{thm:ctm}}
\newtheorem*{pfsp}{Proof of Theorem \ref{thm:specified}}

\theoremstyle{plain}
\newtheorem{THM}{Theorem}
\newtheorem{lemma}{Lemma}[section]
\newtheorem{prop}[lemma]{Proposition}
\newtheorem{thm}{Theorem}[section]
\newtheorem*{claim}{Claim}

\newtheorem*{rmk}{Remark}
\newtheorem{cor}[lemma]{Corollary}

\numberwithin{equation}{section}

\newcommand{\og}{\overline{g}}
\newcommand{\op}{\overline{\pi}}

\newcommand{\hg}{\hat{g}}
\newcommand{\hp}{\hat{\pi}}

\numberwithin{equation}{section}

\newcommand{\R}{\mathbb{R}}

\renewcommand{\div}{\textup{div}}

\newcommand{\bg}{\bar{g}}
\newcommand{\bp}{\bar{\pi}}

\renewcommand{\P}{{\bf P}}
\newcommand{\C}{{\bf C}}
\newcommand{\J}{{\bf J}}

\begin{document}
\date{}%{\today}
\title[Specifying angular momentum and center of mass]
{Specifying angular momentum and center of mass for vacuum initial data sets}
\author{Lan--Hsuan Huang}
\address{Department of Mathematics \\
                 Columbia University\\
                 New York, NY 10027}
\email{lhhuang@math.columbia.edu}
\author{Richard Schoen}
\address{Department of Mathematics \\
                 Stanford University \\
                 Stanford, CA 94305}
\email{schoen@math.stanford.edu}
\author{Mu--Tao Wang}
\address{Department of Mathematics \\
                 Columbia University\\
                 New York, NY 10027}
\email{mtwang@math.columbia.edu}

\thanks{The first author is supported by NSF grant DMS-1005560, the second author is supported by NSF grant DMS-0604960, and the third author is supported by NSF grant DMS-0904281.}
\begin{abstract}
We show that it is possible to perturb arbitrary vacuum asymptotically flat spacetimes to new ones having exactly the same energy and linear momentum, but with center of mass and angular momentum equal to any preassigned values measured with respect to a fixed affine frame at infinity. This is in contrast to the axisymmetric situation where a bound on the angular momentum by the mass has been shown to hold for black hole solutions. Our construction involves changing the solution at the linear level in a shell near infinity, and perturbing to impose the vacuum constraint equations. The procedure involves the perturbation
correction of an approximate solution which is given explicitly.
\end{abstract}

\maketitle

\section{Introduction}
For asymptotically flat spacetimes with appropriate asymptotics there are several conserved quantities which can be measured at spatial infinity. These include the total energy and linear momentum, as well as the angular momentum and center of mass. When we fix an affine frame at infinity the linear and angular momentum as well as the center of mass become three vectors. It is natural to ask whether there are any constraints on these quantities imposed by the Einstein equations. The positive mass theorem provides one such constraint, namely that the energy-momentum vector is a forward pointing timelike vector. In particular this says that the magnitude of the linear momentum vector is bounded above by the energy. For the Kerr solutions which describe rotating stationary axisymmetric vacuum black holes, it is true that the angular momentum must satisfy such a bound. It has been shown over the past several years by S. Dain \cite{D}  and P. T. Chru\'{s}ciel et al. \cite{Ch1,Ch2,Ch3} that such an inequality is also satisfied by general axisymmetric black hole solutions of the Einstein equations. The paper by X. Zhang \cite{Z} proves such an inequality under an energy condition involving his definition of angular momentum density, but it appears that general vacuum data sets do not satisfy this energy condition.

The main results of this paper show that there are no constraints on the angular momentum
and center of mass in terms of the energy-momentum vector for general vacuum solutions
of the Einstein equations. Precisely we fix an affine frame at infinity and we give an effective procedure for adding a specified amount of angular momentum to a solution of the vacuum Einstein equations, producing a new solution with specified angular momentum but with only slightly perturbed energy-momentum vector. We obtain a similar result for the center of mass. Then, by considering a family of initial data near the given one, and by doing the construction continuously, we obtain a perturbation with arbitrarily specified angular momentum and center of mass, while leaving the energy-momentum vector unchanged. One may think of these results  as the pure gravity analogue of the addition to a Newtonian system of a very small symmetrically placed mass far from an axis whose rotation imposes a fixed amount of angular momentum. Similarly one can think of adding a small mass to a Newtonian system which when translated far from the center of mass of the original system produces a fixed change in the center of mass of the new system. From the point of view of the dynamics of the vacuum Einstein
equations we expect the angular momentum that we add near spatial infinity to be
radiated away and to have little effect on the final stationary state of the system. We emphasize that our solutions with arbitrarily specified angular momentum and center mass are complete manifolds, and we can arbitrarily specify the angular momentum of the perturbed data while keeping the energy-momentum and center of mass fixed. Without the completeness condition, there are exterior vacuum solutions with arbitrary prescribed energy-momentum, angular momentum, and center of mass, such as the example of a boosted slice in an exterior Kerr solution computed by Chru\'{s}ciel--Delay \cite{CD}. Also, certain $N$-body solutions constructed by Chru\'{s}ciel--Corvino--Isenberg \cite{CCI} should have large angular momentum, but this comes from orbital angular momentum, i.e. $\vec{c}\times \vec{p}$ for large $\vec{c}$. In particular, the center of mass is not fixed in their case.

From a technical point of view the reason it is possible to make these constructions is that
the angular momentum and center of mass are determined by terms in the expansion of the
solution which are of lower order than those which determine the energy and linear momentum.
The idea then is to make perturbations near infinity which affect only the lower order terms
in the expansion. We do this by explicitly constructing linear perturbations supported in a shell
near infinity which impose the required change in angular momentum (or center of mass),
and then by finding a solution of the vacuum constraint equations which is sufficiently
close to the perturbed system so that the change in angular momentum (or center of mass)
persists. This can be done in such a way that the energy and linear momentum are changed by an arbitrarily small amount.

Let $(M, g,\pi)$ be asymptotically flat in the sense that, outside a compact set, there exists an asymptotically flat coordinate system $\{x^i\}$ so that
\begin{align*}
	 g_{ij} (x) &= \delta_{ij} + O(|x|^{-1})& \pi_{ij}(x) &= O(|x|^{-2}), \\
	 \partial^k g_{ij}(x) &= O(|x|^{-1-k}) \; \; \mbox{ for } k =1,2 & \partial \pi_{ij} (x) &= O(|x|^{-3}).
\end{align*}
In addition, we assume that $(M, g,\pi)$ satisfies the Regge--Teitelboim condition
\begin{align*}
	 g_{ij}(x) - g_{ij}(-x) &= O(|x|^{-2}) & \pi_{ij}(x) + \pi_{ij} (-x) &= O(|x|^{-3}),\\
	 \partial^k (g_{ij}(x) - g_{ij}(-x)) &= O(|x|^{-2-k}) \; \; \mbox{ for } k =1,2 & \partial ( \pi_{ij}(x) + \pi_{ij} (-x) ) &= O(|x|^{-4}).
\end{align*}
The notation $f = O(|x|^{-a})$ means that $|f| \le C |x|^{-a}$ for a constant $C$. We remark that our construction works for data $(g,\pi)$ with weaker assumptions on the decay rates. For simplicity of notation, we assume the decay rates above and do not consider here the question of optimal decay conditions.

Let $E, {\bf C}, {\bf P}, {\bf J}$ denote the energy, center of mass, linear momentum, and angular momentum of $(g,\pi)$. They are defined as limits of integrals over Euclidean spheres
\begin{align*}
	E &= \frac{1}{16 \pi} \lim_{\rho \rightarrow \infty} \int_{|x| = \rho}\sum_{i,j} (g_{ij,i} - g_{ii,j} )\frac{x^j }{|x|} \, d\sigma_0,\\
	{\bf C}^p &= \frac{1}{16 \pi E} \lim_{\rho\rightarrow \infty}\int_{|x|=\rho}\left[ x^p \sum_{i,j}(g_{ij,i}  - g_{ii,j}) \frac{x^j}{|x|} - \sum_i (g_{ip} \frac{x^i}{|x|} - g_{ii} \frac{x^p}{|x|})\right]\, d\sigma_0,
\end{align*}
\begin{align*}
	{\bf P}_i & = \frac{1}{ 8\pi} \lim_{\rho \rightarrow \infty} \int_{|x| = \rho} \sum_j \pi_{ij} \frac{x^j}{|x|} \, d\sigma_0,\\
	{\bf J}_i & = \frac{1}{ 8 \pi E} \lim_{\rho\rightarrow \infty} \int_{|x| = \rho} \sum_{j,k} \pi_{jk} Y_{i}^j \frac{x^k}{|x|} \, d\sigma_0,
\end{align*}
where $d \sigma_0$ is the area measure of the Euclidean sphere $\{|x|=\rho\}$ and $Y_i = \frac{\partial}{\partial x^i} \times \vec{x}$ (cross product) for $i= 1, 2, 3$ are the rotation vector fields. Denote by $\overline{E}, \overline{{\bf C}}, \overline{{\bf P}}, \overline{{\bf J}}$ the energy, center of mass, linear momentum, and angular momentum of $(\og,\op)$.

We now give precise statements of the main theorems, where the definition of the weighted Sobolev spaces $W^{k,p}_{-q}$ is provided in Section \ref{sec:am}, and we assume $p>3$ and $q \in (1/2, 1)$. In our construction, we fix an affine frame near infinity and measure all asymptotic quantities relative to this frame; in fact, we may fix an asymptotically flat coordinate system throughout.

\begin{THM}\label{thm:am}
Let $(g, \pi)$ be a nontrivial vacuum asymptotically flat initial data set with $g = v^4 \delta$ outside a compact set and $\pi(x) + \pi(-x) = O(|x|^{-1-2q})$. Given $\vec{\alpha}\in\mathbb{R}^3$ and $\epsilon >0$, there exists a vacuum asymptotically flat initial data set $(\og, \op)$ such that $(\og, \op)$ is within the $\epsilon$-neighborhood of $(g,\pi)$ in $W^{2,p}_{-q}\times W^{1,p}_{-1-q}$ and
\[
	| \overline{E}- E | \le \epsilon, \quad |\overline{{\bf C}} - {\bf C}| \le \epsilon, \quad \quad |\overline{{\bf P}}-{\bf P}|\le \epsilon,
\]
and
\begin{align}
	|\overline{{\bf J}} -{\bf J}- \vec{\alpha} | \le \epsilon. \label{eq:amineq}
\end{align}
\end{THM}

For the center of mass we prove the following.
\begin{THM}\label{thm:ctm}
Let $(g, \pi)$ be a nontrivial vacuum asymptotically flat initial data set with $g = v^4 \delta$ outside a compact set and $\pi(x) + \pi(-x) = O(|x|^{-1-2q})$. Given $\vec{\gamma}\in\mathbb{R}^3$ and $\epsilon > 0$, there exists a vacuum asymptotically flat initial data set $(\og, \overline{\pi})$ such that $(\og, \op)$ is within the $\epsilon$-neighborhood of $(g,\pi)$ in $W^{2,p}_{-q}\times W^{1,p}_{-1-q}$ and
\[
	| \overline{E}- E| \le \epsilon, \quad |\overline{{\bf J}}-{\bf J}|\le\epsilon,\quad |\overline{{\bf P}}-{\bf P}|\le \epsilon
\]
and
\begin{align}
	|\overline{{\bf C}} - {\bf C} - \vec{\gamma} | \le \epsilon. \label{eq:ctmineq}
\end{align}
\end{THM}
By combining these two results we can change both the center of mass and angular momentum so that they are arbitrarily close to specified values while leaving the energy and linear momentum essentially unchanged. The condition that $g = v^4 \delta$ can be removed by a density theorem. Moreover, given a vacuum initial data set, there is a small perturbation with arbitrary specified angular momentum and center of mass and with the \emph{same} mass and linear momentum.

\begin{THM} \label{thm:specified}
Let $(g,\pi)$ be a nontrivial vacuum initial data set satisfying the Regge--Teitelboim condition. Given any constant vectors $\vec{\alpha}_0, \vec{\gamma}_0 \in \mathbb{R}^3$, there exists a vacuum initial data set $(\bar{g},\bar{\pi})$ within a small neighborhood of $(g, \pi)$ in $W^{2,p}_{-q}\times W^{1,p}_{-1-q}$ and
\[
	\overline{E} = E, \quad \overline{\P} = \P,
\]
and
\[
	\overline{\J} =  \J + \vec{\alpha}_0, \quad \overline{\C} =\C + \vec{\gamma}_0.
\]
\end{THM}

In Section 2 we give an explicit construction of solutions of the linearized
constraint equations which satisfy a certain moment condition.  Sections 3, 4, and 5 are devoted to the proofs of the main theorems. We remark that the constructions of Section 2 are explicit, while the
method of solving the exact constraint equations from the approximate solution involves constructing
a small solution of an elliptic system with leading order term the diagonal Laplace equation. It
should be possible to numerically approximate the resulting solutions to a high degree of accuracy.

%%%%%%%%%%%%%%%%%%%%%%%%%%%%%%%%%%%%%%%%%%%%%%%%%%%%%%%%%%%%%%%%%%%%%%%%%%%%%%
%%%%%%%%%%%%%%%%%%%%%%%%%%%%%%%%%%%%%%%%%%%%%%%%%%%%%%%%%%%%%%%%%%%%%%%%%%%%%%
%%%%%%%%%%%%%%%%%%%                                                                                                                         %%%%%%%%%%%%%%%%%%%
%%%%%%%%%%%%%%%%%%%                            Linearized constraints                                                            %%%%%%%%%%%%%%%%%%%
%%%%%%%%%%%%%%%%%%%                                                                                                                        %%%%%%%%%%%%%%%%%%%
%%%%%%%%%%%%%%%%%%%%%%%%%%%%%%%%%%%%%%%%%%%%%%%%%%%%%%%%%%%%%%%%%%%%%%%%%%%%%%
%%%%%%%%%%%%%%%%%%%%%%%%%%%%%%%%%%%%%%%%%%%%%%%%%%%%%%%%%%%%%%%%%%%%%%%%%%%%%%

\section{Compactly supported solutions of the linearized constraints}
Recall that the vacuum constraint equations for initial data $(g,\pi)$ may be written
\[
	R(g)+\frac{1}{2}\left( \textup{Tr}_g\pi \right)^2-|\pi|^2= 0,\quad \div_g(\pi)=0,
\]
where $\pi_{ij}=K_{ij}-\textup{Tr}_g(K) g_{ij}$ is the momentum tensor. In general, we consider the constraint map $\Phi$ defined by
\[
	\Phi (g, \pi)=(R(g)+\frac{1}{2}\left( \textup{Tr}_g\pi \right)^2-|\pi|^2, \div_g (\pi)).
\]
The vacuum constraint equations, linearized at the trivial data $(\delta,0)$,
become
\[
	L\sigma:=\sum_{i,j}(\sigma_{ij,ij}-\sigma_{ii,jj})=0,\quad \div(\tau)=0,
\]
for symmetric $(0,2)$ tensors $(\sigma,\tau)$.
In this section we will construct solutions of the linearized constraint equations which are
compactly supported in the shell $A_1=\{x:\ 1<|x|<2\}$ contained in $\mathbb{R}^3$
and which have certain specified moment conditions with respect to rotation vector fields. We use the Einstein summation convention and sum over repeated indices; though, sometimes we employ summation symbols for clarity.

We write the Euclidean metric on $A_1$ in spherical coordinates $dr^2+r^2 \tilde{g}_{ab}dx^a dx^b$ where $\tilde{g}$ is the standard round metric on $S^2$. The coordinates are labeled by $r=x^0$ and $\theta, \phi=x^1,x^2$. The ranges for the indices are $i, j, k, l,...=0, 1, 2$, and $a, b, c, d, e, ...= 1, 2$.
If $\alpha,\beta$ are one-forms, we define the symmetric product $\alpha\odot\beta$ to be
the symmetric $(0,2)$ tensor whose components are
\[
	(\alpha\odot\beta)_{ij}=\frac{1}{2}(\alpha_i\beta_j+\alpha_j\beta_i).
\]
We first impose the following ansatz for our solutions.

\begin{lemma}\label{ansatz}
Suppose that $q$ and $Q$ are functions of $r$, $\tilde{\alpha}$ is a one-form on $S^2$, and $\tilde{\tau}$ is a trace-free symmetric $(0,2)$ tensor on $S^2$. Then $\tau=2q\tilde{\alpha}\odot dr+ Q \tilde{\tau}$ is a trace-free symmetric $(0,2)$ tensor on $A_1$. The condition $\div(\tau)=0$ becomes
 \[
 	\widetilde{\div}\tilde{\alpha}=0 \quad \text{ and } \quad (r^2q)'\tilde{\alpha} + Q(\widetilde{\div}{\tilde{\tau}})=0,
 \]
 where $\widetilde{\div}$ is the divergence operator of $S^2$ on tensors.

 Suppose that $p$ and  $P$ are functions of $r$, $\tilde{\eta}$ is a one-form on $S^2$, and $\tilde{\sigma}$ is a trace-free symmetric $(0,2)$ tensor on $S^2$.
Then $\sigma=2p\tilde{\eta}\odot dr+ P \tilde{\sigma}$ is a trace-free symmetric $(0,2)$ tensor on $A_1$,  and $L \sigma =0$ if
\[
	2r(r p)' \widetilde{\div}\tilde{\eta}+ P(\widetilde{\div}\widetilde{\div} \tilde{\sigma})=0.
\]	
\end{lemma}
This lemma follows directly from the following two computational lemmas.
For a coordinate system $x^i$ on $\R^3$, denote by $g_{ij} dx^i dx^j$ the Euclidean metric.
\begin{lemma} Let $h$ be any symmetric $(0,2)$ tensor on $\R^3$, then
\[
	(\div h)_i=\frac{1}{\sqrt{g}}\frac{\partial}{\partial x^k}(\sqrt{g} h_i^k)+\frac{1}{2} h_{jk} \frac{\partial }{\partial x^i}(g^{jk}).
\]
\end{lemma}

\begin{proof} Let $V^i$ be any vector field. We have $V^i(h_i^k)_{;k}=( V^i h_i^k)_{;k}-h_i^k V^i_{;k}$. Now
\[
	( V^i h_i^k)_{;k}=\frac{1}{\sqrt{g}}\frac{\partial}{\partial x^k}(V^i h_i^k\sqrt{g})=V^i\frac{1}{\sqrt{g}}\frac{\partial}{\partial x^k}(\sqrt{g} h_i^k)+ h_i^k \frac{\partial V^i}{\partial x^k}.
\]
Therefore $ V^i (h_i^k)_{;k}= V^i\frac{1}{\sqrt{g}}\frac{\partial}{\partial x^k}(\sqrt{g} h_i^k)- h_i^k \Gamma_{kl}^i V^l$ and then
\[
	(\div h)_i=\frac{1}{\sqrt{g}}\frac{\partial}{\partial x^k}(\sqrt{g} h_i^k)- h_l^k \Gamma^l_{ki}.
\]
Lastly, we plug in the formula for $\Gamma_{ij}^k$.
\end{proof}

Now we apply this to the spherical coordinates
\[
g_{ij} dx^i dx^j=dr^2+r^2d\theta^2+r^2\sin^2\theta d\phi^2=dr^2+r^2 \tilde{g}_{ab} dx^a dx^b
\]
and write $h=h_{00} dr^2+2h_{0a} dx^a dr+h_{ab} dx^a dx^b$ and $\alpha=\alpha_0 dr+\alpha_a dx^a$.

\begin{lemma} Let $h$ be a symmetric $(0,2)$ tensor and $\alpha$ a one-form on $\R^3$, then
\[
	\begin{split}\div h&=[ r^{-2} \frac{\partial}{\partial r}(r^2h_{00})+r^{-2} \frac{1}{\sqrt{\tilde{g}}}\frac{\partial}{\partial x^a}(\sqrt{\tilde{g}} \tilde{g}^{ab} h_{0b})-r^{-3} \tilde{g}^{ab} h_{ab}] dr\\
	&\quad +\{ r^{-2} \frac{\partial}{\partial r} (r^2 h_{a0})+ r^{-2}[ \frac{1}{\sqrt{\tilde{g}}}\frac{\partial}{\partial x^b}(\sqrt{\tilde{g}} \tilde{g}^{bc} h_{ac})+\frac{1}{2}h_{bc}\frac{\partial}{\partial x^a} \tilde{g}^{bc}]\} dx^a\end{split}
\]
and
\[
	\div \alpha=r^{-2} \frac{\partial}{\partial r}(r^2\alpha_0)+r^{-2}\frac{1}{\sqrt{\tilde{g}}}\frac{\partial}{\partial x^b}(\sqrt{\tilde{g}} \tilde{g}^{ab} \alpha_{a}).
\]
\end{lemma}

We shall take the following ansatz on $h$ that $h_{00}=0$ and $\tilde{g}^{ab} h_{ab}=0$. We also assume that $h_{0a}=p(r) \tilde{\eta}_a$ and $h_{ab}=P(r)\tilde{h}_{ab}$ for a one-form $\tilde{\eta}$ and a symmetric $(0,2)$ tensor $\tilde{h}$ on $S^2$; that is, we have

\[h=2p(r)\tilde{\eta}\odot dr+ P(r)\tilde{h}.\]
We then have
\[
	\div h=(r^{-2} p) ( \widetilde{\div}\tilde{\eta}) dr+r^{-2} (r^2 p)'\tilde{\eta}+r^{-2} P
\widetilde{\div}{\tilde{h}}
\]
and
\[
	\div \div h=2 r^{-3} (rp)' (\widetilde{\div}\tilde{\eta})+r^{-4} P (\widetilde{\div}\widetilde{\div} \tilde{h}).
\]
Hence, Lemma \ref{ansatz} follows directly from the above two identities.

In particular, if we consider the tensor $\tau$ of the following expression
\begin{equation}\label{ansatz_tau}
	\tau=-2q \widetilde{\div}(\tilde{\tau})\odot dr+(r^2q)'{\tilde{\tau}}.
\end{equation}
Then $\tau$ satisfies $\div \tau=0$ on $A_1$ if
\begin{equation}\label{eq_tau}
	\widetilde{\div} \widetilde{\div}{\tilde{\tau}}=0\quad \mbox{on } S^2.
\end{equation}
Also, we consider $\sigma$ as follows:
\begin{equation}
	\label{ansatz_sigma}\sigma=-2p\widetilde{\div}\tilde{\sigma}\odot dr+ 2r(r p)'\tilde{\sigma}.
\end{equation}
Then $\sigma$ satisfies $L\sigma=0$ on $A_1$ for any trace-free symmetric $(0,2)$ tensor $\tilde{\sigma}$ on $S^2$.

In the following, we show that there are nontrivial solutions to \eqref{eq_tau}. We can start from any one-form $\tilde{\eta}$ on $S^2$ and construct a trace-free symmetric $(0,2)$ tensor $\tilde{\delta}^* \tilde{\eta}$ on $S^2$ by defining
\[(\tilde{\delta}^* \tilde{\eta})_{ab}=\frac{1}{2}(\tilde{\eta}_{a;b}+\tilde{\eta}_{b;a}-\tilde{g}^{cd}\tilde{\eta}_{c;d} \tilde{g}_{ab}).\]

We need the following computational result.
\begin{lemma}\label{div}
\[
	\widetilde{\div}\tilde{\delta}^*\tilde{\eta}=-\frac{1}{2}(d d^*+d^* d)\tilde{\eta}+\tilde{\eta}.
\]
In particular, if $u$ is a function on $S^2$, then
\[
	\widetilde{\div}\tilde{\delta}^*du=\frac{1}{2}d( \Delta u+2u),
\]
and
\[
	\widetilde{\div}\tilde{\delta}^*(*du)=\frac{1}{2}*d( \Delta u+2u).
\]
\end{lemma}

\begin{proof}
We recall that for one-forms on $S^2$, we have $d^*= -*d*$, and for a function $u$,  $d^*du=-\Delta u$. We assume $\tilde{\eta}=du$ and compute in an orthonormal frame:
\[
	(\widetilde{\div}(\tilde{\delta}^*du))_1=\frac{1}{2}(u_{1;11}-u_{2;21})+ u_{1;22}.
\]
On the other hand, $(\frac{1}{2}(d d^*+d^* d) du)_1=-\frac{1}{2} (\Delta u)_1=-\frac{1}{2}(u_{1;11}+u_{2;21})$. Therefore,
\[
	(\widetilde{\div}\tilde{\delta}^*du)_1+(\frac{1}{2}(d d^*+d^* d) du)_1=u_{1;22}-u_{2;21}.
\]
At last, we use the commutation formula $u_{b;ab}=u_{b;ba}+u_a$ for $a\neq b$ on $S^2$. Other formulae can be checked similarly.
\end{proof}

The next lemma shows how solutions of Equation (\ref{eq_tau}) can be constructed:

\begin{lemma}\label{div_eq}
Suppose that $\tilde{\alpha}$ is any one-form on $S^2$. Then $\tilde{\tau}=\tilde{\delta}^*\tilde{\alpha}$ is a trace-free symmetric $(0,2)$ tensor on $S^2$. Moreover, if $\tilde{\alpha}=*du$ for a function $u$ on $S^2$, then $\tilde{\tau}$ satisfies $\widetilde{\div}  \widetilde{\div}\tilde{\tau}=0$.
\end{lemma}

\begin{proof}
This follows from the formula of Lemma \ref{div},
\[
	\widetilde{\div}\tilde{\tau}=\frac{1}{2}*d(\Delta u+2u),
\]
together with the fact that $(d^*)^2=0$.
\end{proof}
We need the following computational result.
\begin{lemma}\label{integral}
Suppose that $\tau=2q\tilde{\alpha}\odot dr+ Q \tilde{\tau}$, $\sigma=2p\tilde{\eta}\odot dr+ P \tilde{\sigma}$, and $Y=Y^a\frac{\partial}{\partial x^a}$ is tangent to $S^2$. Then
\[
	\begin{split}
	&\left(\frac{1}{2} \tau_{ij,l} Y^l +\tau_{il} Y^l_{,j} \right)\sigma^{ij}\\
	&=pq r^{-2} (\tilde{\alpha}_{b;a}Y^a+\tilde{\alpha}_a Y^a_{;b}) \tilde{\eta}_c\tilde{g}^{bc}
+\frac{1}{2}PQ r^{-4}(\tilde{\tau}_{bc;a}Y^a+\tilde{\tau}_{ab} Y^a_{;c}+\tilde{\tau}_{ac}Y^a_{;b})\tilde{g}^{bd}\tilde{\sigma}_{de} \tilde{g}^{ec},
	\end{split}
\]
where $\tilde{\alpha}_{b;a}$, $\tilde{\tau}_{bc;a}$, and $Y^a_{;b}$ denote covariant derivatives of $\tilde{\alpha}$, $\tilde{\tau}$, and $Y$ with respect to the standard metric $\tilde{g}_{ab}$ on $S^2$.
\end{lemma}
\begin{proof}
Direct computation.
\end{proof}

We are finally in a position to prove the main results of this section.
\begin{thm} \label{thm:localpert}
Given any $\vec{\lambda}=(\lambda_1,\lambda_2,\lambda_3)$, there exist symmetric $(0,2)$ tensors $\sigma, \tau \in C^{\infty}_0(A_1)$ with $\sigma(x) = \sigma(-x)$ and  $\tau(x) = \tau(-x)$ satisfying
\begin{align}
	L \sigma &=0 \notag \\
	\sum_i \tau_{ij,i} & = 0, \quad \mbox{for } j = 1,2,3,\label{eq:tau1}
\end{align}
so that
\begin{align} \label{eq:angular}
	\int_{A_1}  \left[\frac{1}{2} \tau_{ij,l}(Y_k)^l +  \tau_{il} (Y_k)^l_{,j}\right] \sigma^{ij}\, d x = \lambda_k
\end{align}
for $k=1,2,3$ where $Y_k =  \frac{\partial}{\partial x^k}\times \vec{x}$ (cross product).
\end{thm}
\begin{rmk}
By a direct computation, the integrand in \eqref{eq:angular} equals  $\frac{1}{2} (\mathfrak{L}_{Y_k} \tau)_{ij} \sigma^{ij}$, where $\mathfrak{L}_{Y_k} \tau$ is the Lie derivative of $\tau$ along $Y_k$ on $\mathbb{R}^3$. Thus, if $\tau$ is axisymmetric with respect to $Y_k$, i.e. $\mathfrak{L}_{Y_k} \tau = 0$, then \eqref{eq:angular} is always zero.
\end{rmk}
\begin{proof}
We first show how to make the integral on the left of (\ref{eq:angular}) for $k=1$ nonzero. To simplify notation for this purpose we denote
$Y_1$ by $Y$. We choose $\tau$ of the form (\ref{ansatz_tau}) for some $\tilde{\tau}=\tilde{\delta}^**du$, where $u$ is an even function on $S^2$. By Lemma \ref{ansatz} and Lemma \ref{div_eq}, $\tau$ satisfies Equation (\ref{eq:tau1}) and has the desired symmetry $\tau(x)=\tau(-x)$.  We take $\sigma$ of the form (\ref{ansatz_sigma}) for a symmetric $(0,2)$ tensor $\tilde{\sigma}$ to be determined later. By Lemma \ref{integral}, the integral in question can be written as
\[
	\begin{split}
		&(\int_1^2 pq r^{-2} dr) \int_{S^2} (\tilde{\alpha}_{b;a}Y^a+\tilde{\alpha}_a Y^a_{;b})\tilde{\eta}_c\tilde{g}^{bc} dS^2\\
		&+(\int_1^2 \frac{1}{2}PQ r^{-4} dr) \int_{S^2} (\tilde{\tau}_{bc;a}Y^a+\tilde{\tau}_{ab} Y^a_{;c}+\tilde{\tau}_{ac}Y^a_{;b})\tilde{g}^{bd}\tilde{\sigma}_{de} \tilde{g}^{ec} dS^2,
	\end{split}
\]
for
$\tilde{\alpha}=-\widetilde{\div}\,\tilde{\tau}$, $\tilde{\eta}=-\widetilde{\div}\,\tilde{\sigma}$, $P=2r(rp)'$, and $Q=(r^2q)'$. We can choose $p$ and $q$ to be any compactly supported functions on the interval $(1, 2)$ so that $(\int_1^2 pq r^{-2} dr)$ and $(\int_1^2 \frac{1}{2}PQ r^{-4}dr)$ are arbitrary.  It suffices to choose $\tilde{\sigma}$ to make the following integral nonzero
\[
	\int_{S^2} (\tilde{\tau}_{bc;a}Y^a+\tilde{\tau}_{ab} Y^a_{;c}+\tilde{\tau}_{ac}Y^a_{;b})\tilde{g}^{bd}\tilde{\sigma}_{de} \tilde{g}^{ec} dS^2 \neq 0.
\]
To achieve this, we take
\[
	\tilde{\sigma}_{bc}=\tilde{\tau}_{bc;a}Y^a+\tilde{\tau}_{ab} Y^a_{;c}+\tilde{\tau}_{ac}Y^a_{;b}
\]
and define $\sigma$ by Equation (\ref{ansatz_sigma}). Since $Y$ is invariant under $x\mapsto -x$,
$\sigma$ defined in this way has the desired symmetry and satisfies $L\sigma=0$. It is not hard to check that $\tilde{\sigma} =\mathfrak{L}_Y\tilde{\tau}$, the Lie derivative of $\tilde{\tau}$ with respect to $Y$ on $S^2$. Thus, we can take an even function $u$ (e. g. the restriction of any homogeneous polynomial of even degree) so that $\tilde{\tau}_{bc;a}Y^a+\tilde{\tau}_{ab} Y^a_{;c}+\tilde{\tau}_{ac}Y^a_{;b}$ is nonzero.

To achieve the desired conclusion, we consider the linear functional $T_{(\sigma,\tau)}(v)$ given by the left-hand side of (\ref{eq:angular}) with vector field $Y=v\times\vec{x}$. Since this is a nonzero linear functional we may choose a positively oriented orthonormal basis $\{e_1,e_2,e_3\}$ so that the vector $(T_{(\sigma,\tau)}(e_1),T_{(\sigma,\tau)}(e_2),T_{(\sigma,\tau)}(e_3))$ is proportional to $\vec{\lambda}$, and after multiplication of $\tau$ by a constant we may assume the vector is equal to $\vec{\lambda}$. It follows that there is a rotation $R$ of $\mathbb{R}^3$ so that $T_{(\sigma,\tau)}(R(\frac{\partial}{\partial x^k}))=\lambda_k$ for $k=1,2,3$. It follows that (\ref{eq:angular}) holds for the pair $((R^{-1})^*(\sigma),(R^{-1})^*(\tau))$ since we clearly have $T_{(S^*\sigma,S^*\tau)}=T_{(\sigma,\tau)}\circ S^{-1}$ for any rotation $S$.

\end{proof}

We will need a corresponding result which will be used to specify the center of mass. This involves
the construction of solutions of $L\sigma=0$ satisfying a moment condition.
\begin{thm} \label{thm:local_cm}
Given any $\vec{\beta}=(\beta_1,\beta_2,\beta_3)\in\mathbb{R}^3$, there exist a trace-free and divergence-free symmetric $(0,2)$ tensor $\sigma \in C^{\infty}_0 (A_1)$ satisfying $L \sigma=0$ so that
\begin{align} \label{eq:cm}
	\int_{A_1} x^p \sum_{i,j,k}(\sigma_{ij,k})^2\, d x = \beta_p
\end{align}
for $p=1,2,3$.
\end{thm}
\begin{proof} We first show how to make the integral on the left of (\ref{eq:cm}) nonzero
for $p=1$. By starting with a nonzero function $u$ supported in the first octant, we find from Lemma \ref{div_eq} and \eqref{ansatz_tau} a nonzero trace-free symmetric (0,2) tensor $\sigma$ in $C^{\infty}_0(A_1)$ satisfying $\div (\sigma) =0$. Then, in particular $L \sigma =0$.  Because $\sigma$ is supported in the first octant, this implies
\[
	\int_{A_1} x^1 \sum_{i,j,k}(\sigma_{ij,k})^2\, d x>0.
\]

%By starting with a nonzero trace-free symmetric $(0,2)$ tensor $\tilde{\sigma}$
%supported in the first octant we find from (\ref{ansatz_sigma}) a nonzero $\sigma$ in
%$C^\infty_0(A_1)$ satisfying $L\sigma=0$ and with support in the first octant. This implies
%$\int_{A_1} x^1 \sum_{i,j,k}(\sigma_{ij,k})^2\, d x>0$.

We then let $T_\sigma$ be the linear functional on $\mathbb{R}^3$ given by
\[ T_\sigma(v)=\int_{A_1} x\cdot v \sum_{i,j,k}(\sigma_{ij,k})^2\, d x.
\]
Since $T_\sigma$ is nonzero, there is a positively oriented orthonormal basis $\{e_1,e_2,e_3\}$  for which the vector defined by $(T_\sigma(e_1),T_\sigma(e_2),T_\sigma(e_3))$ is proportional to $\vec{\beta}$. Replacing
$\sigma$ with a scalar multiple we may assume that $T_\sigma(e_p)=\beta_p$ for $p=1,2,3$.
Thus there is a rotation $R$ with $R(\frac{\partial}{\partial x^p})=e_p$ so that
$T_\sigma(R(\frac{\partial}{\partial x^p}))=\beta_p$ for $p=1,2,3$.
Since for a rotation $S$ we have $T_{S^*\sigma}=T_\sigma\circ S^{-1}$ it follows that
$(R^{-1})^*\sigma$ satisfies the required condition (\ref{eq:cm}).
\end{proof}

%%%%%%%%%%%%%%%%%%%%%%%%%%%%%%%%%%%%%%%%%%%%%%%%%%%%%%%%%%%%%%%%%%%%%%%%%%%%%%
%%%%%%%%%%%%%%%%%%%%%%%%%%%%%%%%%%%%%%%%%%%%%%%%%%%%%%%%%%%%%%%%%%%%%%%%%%%%%%
%%%%%%%%%%%%%%%%%%%                                                                                                                         %%%%%%%%%%%%%%%%%%%
%%%%%%%%%%%%%%%%%%%                                            Angular momentum                                                 %%%%%%%%%%%%%%%%%%%
%%%%%%%%%%%%%%%%%%%                                                                                                                         %%%%%%%%%%%%%%%%%%%
%%%%%%%%%%%%%%%%%%%%%%%%%%%%%%%%%%%%%%%%%%%%%%%%%%%%%%%%%%%%%%%%%%%%%%%%%%%%%%
%%%%%%%%%%%%%%%%%%%%%%%%%%%%%%%%%%%%%%%%%%%%%%%%%%%%%%%%%%%%%%%%%%%%%%%%%%%%%%

\section{Specifying the angular momentum} \label{sec:am}
We first state a general analytic result which constructs new initial data sets from given ones. Let $L \sigma := \sum_{i,j} (\sigma_{ij,ij} - \sigma_{ii,jj} )$ be the linearized (at the Euclidean metric) scalar curvature map which goes from smooth symmetric $(0,2)$ tensors to smooth functions. Denote by $W^{k,p}_{-q}$ the weighted Sobolev spaces defined as follows. We say $ f \in W^{k,p}_{-q}$ if
$$
		\| f\|_{W^{k,p}_{-q}} := \left( \int_M \sum_{|\alpha| \le k} \left( \big| D^{\alpha} f \big| \rho ^{ | \alpha| + q }\right)^p \rho^{-3} \, d\textup{vol}_g  \right)^{\frac{1}{p}} < \infty,
$$
where $\alpha$ is a multi-index and $\rho$ is a continuous function with $\rho = |x|$ on the region where the asymptotically flat coordinate system $\{x^i\}$ is defined. When $p = \infty$,
$$
		\| f\|_{W^{k,\infty}_{-q}} = \sum_{|\alpha | \le k} ess \sup_{M } | D^{\alpha} f | \rho^{ |\alpha| + q}.
$$
We assume $k=1$ or $2$, $q\in (1/2, 1)$, and $p>3$. By our assumption $(g, \pi) \in W^{2,p}_{-q} \times W^{1,p}_{-1-q}$.

 In the following, the notation $f= O(r^{-a})$ means that $|f|\le Cr^{-a}$, $|\partial f| \le C r^{-a-1}$ for some constant $C$ independent of $k$ and the analogous conditions on successive derivatives as needed. We remark that the following results would hold similarly for any $p>3/2$, but the ``$O$''-notation below would mean the decay in weighted Sobolev norms. Here we assume $p>3$ because the weighted Sobolev norms can be replaced by pointwise estimates using the Sobolev  imbedding theorem.

%Let $L^*$ be the formal adjoint of $L$ so that $L^*$ is the operator that maps smooth functions to symmetric $(0,2)$--tensors
%\[
%		( L^* f )_{ij} = f_{,ij} - ( \Delta f)\delta_{ij}.
%\]
Let $r:= |x|$. Denote the shell by $A_k = \{ k< r < 2k\}$ and denote by $C^{\infty}_0 (A_k)$ the set of functions or tensors which are compactly supported in $A_k$.

\begin{prop} \label{prop:local}
Let the symmetric $(0,2)$ tensors $\sigma, \tau \in C^{\infty}_0(A_1)$ satisfy the linearized constraint equations
\begin{align*}
		L \sigma &= 0 \\%\label{eq:lconst1}
		\sum_i \tau_{ij,i} &= 0 \quad \mbox{for } j = 1,2,3. \notag
\end{align*}
Define $\sigma^k, \tau^k \in C^{\infty}_0(A_k)$ by
\begin{align*} %\label{def:scale}
		\sigma^k = k^{-1} \sigma(x/k), \quad \mbox{and } \tau^k = k^{-2} \tau(x/k).
\end{align*}
Then given any vacuum initial data set $(M, g, \pi)$ with decay rate $g-\delta=O(r^{-1})$,
$\pi=O(r^{-2})$, and any fixed $q\in (1/2, 1)$, there exists a sequence of vacuum initial data sets $\{(\og^k, \op^k)\}$ so that for $k$ large and  outside a fixed compact set (independent of $k$),
\begin{align}
	\og^k_{ij} &= \left( 1 +  \frac{A^k}{r}\right) g_{ij}+ \sigma^k_{ij} + O(r^{-2q}) \label{eq:gasym}\\
	\op^k_{ij} &= \pi_{ij} + \tau^k_{ij} + \frac{1}{r^3}\left[ -B^k_i x_j - B^k_j x_i + \sum_l B^k_lx_l \delta_{ij} \right] + O(r^{-1-2q}), \label{eq:piasym}
\end{align}
where $A^k$ and $(B^k_1, B^k_2, B^k_3)$ are constants.  The initial data sets $\{(\og^k, \op^k)\}$ are small perturbations of $(g,\pi)$ in the weighted Sobolev spaces; in fact,
\[
	\| g - \og^k \|_{W^{2,p}_{-q}} \rightarrow 0, \quad \| \pi - \op^k \|_{W^{1,p}_{-1-q}} \rightarrow 0, \quad \mbox{ as } k \rightarrow \infty.
\]
Moreover,
\begin{align} \label{eq:masslm}
	\overline{E}_k \rightarrow E \quad \mbox{and} \quad \overline{{\bf P}}_k \rightarrow {\bf P} \qquad \mbox{as } k\rightarrow \infty.
\end{align}
\end{prop}
\begin{pf}
Let $\hg^k = g+ \sigma^k$ and $\hp^k = \pi + \tau^k$. Then $(\hg^k, \hp^k)$ satisfies the constraint equations $\Phi(\hg^k, \hp^k ) = (0,0)$ everywhere in $M \setminus A_k$ and $\Phi (\hg^k, \hp^k) = (O(r^{-4}), O(r^{-4}))$ in $A_k$. We denote
\[
	(\mathcal{L}_g {\bf X})_{ij} = {\bf X}_{i;j} + {\bf X}_{j;i} - (\div_g {\bf X})g_{ij}
\]
for any vector field $X$ and metric $g$. By the proof of \cite[Theorem 1]{CS06}, there exist $(u^k, {\bf X}^k)$ on $M$ and $(h^k, w^k)$ with compact supports (uniformly in $k$) such that
\[
	\og^k = (u^k)^4 \hg^k + h^k, \quad \mbox{and } \quad \op^k= (u^k)^2(\hp^k + \mathcal{L}_{\hg}{ \bf X}^k) + w^k
\]
satisfy $\Phi(\og^k, \op^k) = 0$ for all $k$ large, and
\begin{align} \label{eq:decay}
	\| u^k - 1 \|_{W^{2,p}_{-q}} \rightarrow 0, \quad \| {\bf X}^k \|_{W^{2,p}_{-q}}\rightarrow 0,\quad \| h^k \|_{W^{2,p}_{-q}} \rightarrow 0, \quad \| w^k \|_{W^{1,p}_{-q}} \rightarrow 0,
\end{align}
as $k\rightarrow \infty$. The constraint equations imply $\Delta u^k = O(r^{-2-2q})$ and $\Delta ({\bf X}^k)_i = O(r^{-2-2q})$. It follows that
\begin{align*}
			u^k = 1 + \frac{A^k}{4r} + O(r^{-2q}), \quad \mbox{and} \quad ({\bf X}^k)_i = \frac{B^k_i}{r} + O(r^{-2q}).
\end{align*}
Therefore, \eqref{eq:gasym} and \eqref{eq:piasym} follow, and the convergence of \eqref{eq:masslm} can be derived as in \cite{CS06}.
\qed
\end{pf}

For the rest of the section, we consider the special case of Proposition \ref{prop:local} when $g = v^4 \delta$ outside a compact set so we have $g = \left(1 + \frac{2 E}{r} \right)\delta_{ij} + O(r^{-2})$.  We also assume that $\pi =O(r^{-2})$, $\pi(x)+\pi(-x)=O(r^{-1-2q})$. %{\color{blue}(The reason that we write the decay rate $-1-2q$ instead of $-3$ as stated in Theorem \ref{thm:am} is that in the proof of Theorem \ref{thm:tilde} we would apply the estimates to this class of data sets.) }
That any initial data can be approximated by such data follows from \cite{CS06}. Note that the asymptotic oddness condition is the Regge--Teitelboim condition  \cite{RT,BO} required for the existence of ${\bf J}$. Therefore, near infinity, $(\og^k, \op^k)$ satisfy the conditions
\begin{align}
	\og^k_{ij} &= (u^k)^4(g_{ij}+\sigma^k_{ij}) =\left( 1 + \frac{2E+A^k}{r} \right)\delta_{ij} + \sigma^k_{ij} + O(r^{-2q}) \label{eq:gasym2}\\
	\op^k_{ij} &= (u^k)^2 (\pi+\tau^k+\mathcal{L}_{\hg}{\bf X}^k)_{ij} \notag\\
	&= \pi_{ij}+\tau^k_{ij} + \frac{1}{r^3}\left[ -B^k_i x_j - B^k_j x_i + \sum_l B^k_lx_l \delta_{ij} \right] + O(r^{-1-2q}). \label{eq:piasym2}
\end{align}

\begin{pfam}
Let $(g, \pi)$ be a vacuum initial data set satisfying the above conditions. Choose $\sigma$ and $\tau$ satisfying the assumptions in Theorem \ref{thm:localpert} with $\vec{\lambda}=-8\pi\vec{\alpha}$. There exist vacuum initial data sets $(\og^k, \op^k)$ satisfying \eqref{eq:gasym2} and \eqref{eq:piasym2} by Proposition \ref{prop:local}, and $|\overline{E_k} - E|\le \epsilon$ and $|\overline{\bf P}_k - {\bf P} | \le \epsilon$. It remains to prove the desired properties of the center of mass and angular momentum. In the following, we suppress the superscript $k$ of $\og^k$ and $\op^k$ whenever it is clear from the context.

We assume that $\vec{\alpha}=(\alpha_1,\alpha_2,\alpha_3)$. We denote by $Y_p$ the Euclidean Killing
vector field $Y_p=\frac{\partial}{\partial x^p}\times\vec{x}$ for $p=1,2,3$. Because of the asymptotics of $(\og, \op)$,
\begin{align*}
	\overline{E} \, \overline{{\bf J}}_p &= \frac{1}{8\pi} \lim_{\rho\rightarrow \infty} \int_{r = \rho} \sum_{i,j} \op_{ij}(Y_p)^i \frac{x_j}{r} \, d\sigma_0\\
	&=  \frac{1}{8\pi} \lim_{\rho\rightarrow \infty} \int_{r = \rho} \op_{ij}(Y_p)^i \nu^j \, d\sigma_{\og},
\end{align*}
where $\nu$ and $d\sigma_{\og}$ are respectively the outward unit normal vector and the area measure of $\{r = \rho\}$ with respect to $\og$. To shorten notation we do the following estimates using $Y$ in place of $Y_p$. By the divergence theorem, assuming $\rho_0 \ll k$ and $2k < \rho$,
\begin{align*}
	\int_{r = \rho} \op_{ij}Y^i \nu^j \, d\sigma_{\og} = \int_{ \rho_0 \le r \le \rho} g^{lj} (\op_{il}Y^i)_{;j} \, d\textup{vol}_{\og} + \int_{r = \rho_0} \op_{ij}Y^i \nu^j \, d\sigma_{\og}.
\end{align*}
Because $\tau^k$ vanishes on $\{ r = \rho_0\}$,
\begin{align*}
		 \int_{r = \rho_0}\op_{ij}Y^i \nu^j \, d\sigma_{\og} = \int_{r=\rho_0} \sum_{i,j} (u^k)^2 (\pi+\mathcal{L}_{\hg} {\bf X}^k)_{ij} Y^i \nu^j \, d\sigma_{\og}.
\end{align*}
Although the absolute value of the last integrand is $O(\rho_0^{-1})$ which would not have a finite limit, the asymptotic symmetry conditions say that $\pi$ is odd and the term $\mathcal{L}_{\hg} {\bf X}^k$ is also asymptotically odd,  and hence the leading order term of the integrand is odd causing the limit to be finite. In fact, we have for $k$ large enough,
\begin{align}
	 &\int_{r=\rho_0} \sum_{i,j} (u^k)^2 (\pi+\mathcal{L}_{\hg} {\bf X}^k)_{ij} Y^i \nu^j \, d\sigma_{\og}  \notag\\
	 &=\int_{r=\rho_0} \sum_{i,j} \pi_{ij} Y^i \nu^j \, d\sigma_{\og} +\int_{r=\rho_0} \sum_{i,j} (\mathcal{L}_{\hg} {\bf X}^k)_{ij} Y^i \nu^j \, d\sigma_{\og}\label{eq:1}\\
	 &\quad +  \int_{r=\rho_0} \sum_{i,j} [(u^k)^2-1] (\pi+\mathcal{L}_{\hg} {\bf X}^k)_{ij} Y^i \nu^j \, d\sigma_{\og} \label{eq:2}\\
	 &=8\pi E {\bf J}_p+ O(\rho_0^{1-2q}).\notag
\end{align}
Clearly, the first integral in \eqref{eq:1} is $8\pi E {\bf J}_p+ O(\rho_0^{-1})$. For the second integral in \eqref{eq:1}, we use \eqref{eq:decay} and choose $k$ large so that $|\mathcal{L}_{\hg} {\bf X}^k|$ is small, say less than $\rho_0^{-4}$. The integral in \eqref{eq:2} is $O(\rho_0^{1-2q})$ by \eqref{eq:gasym2}, \eqref{eq:piasym2}, and the asymptotic symmetry of $\pi$.

To estimate the interior integral, by the constraint equation $\og^{lj}\op_{il;j}= 0$ and the condition that $Y$ is a Euclidean Killing vector field,
\begin{align} \label{eq:pointwise}
		\og^{lj}(\op_{il}Y^i)_{;j} = (\og^{lj} -\delta^{lj}) \op_{il} Y^i_{,j} + \og^{lj} \op_{il} Y^s \overline{\Gamma}_{js}^i.
%	 &\int_{ k \le r \le \rho} \sum_{i,j} (\op_{ij}Y^i)_{;j} \, d\textup{vol}_{\og} \\
%	 &= \int_{ k \le r \le \rho} \sum_{i,j,l} (\og^{jl} -\delta^{jl}) \op_{li} Y^i_{,j} \, d \textup{vol}_{\og} +  \int_{ k \le r \le \rho} \sum_{i,j,l,p} \og^{jl} \op_{li} Y^p \overline{\Gamma}_{jp}^i \, d\textup{vol}_{\og}.
\end{align}
By \eqref{eq:gasym2}, \eqref{eq:piasym2}, and $\sigma^k(x) = \sigma^k(-x)$, the integral of the first term on the right-hand side is
\begin{align*}
	 \int_{ \rho_0 \le r \le \rho}  (\og^{lj} -\delta^{lj}) \op_{il} Y^i_{,j} \, d \textup{vol}_{\og} &= - \int_{ A_k} \sum_{i,j, l} \sigma^k_{ij} \tau^k_{il}Y^l_{,j} \, dx + O(\rho_0^{1-2q}) \\
	& = - \int_{ A_1} \sum_{i,j, l} \sigma_{ij} \tau_{il}Y^l_{,j} \, dx + O(\rho_0^{1-2q}),
\end{align*}
where we use $\pi(x) + \pi(-x) = O(r^{-1-2q})$ and $\sigma^k(x) = \sigma^k(-x)$ to estimate the error terms. For example, for some of the error terms,
\begin{align*}
	& \int_{\rho_0 \le r \le \rho} - \frac{2 E A^k}{r} \sum_{i,j,l} \delta_{lj} \pi_{il} Y^i_{,j} \, d\mbox{vol}_{\og} = \int_{\rho_0}^{\rho} O(r^{-2-2q}) \, r^2 dr = O(\rho_0^{1-2q}),\\
	& \int_{\rho_0 \le r \le \rho} - \sum_{i,j,l}\sigma^k_{lj} \pi_{il} Y^i_{,j} \, d\mbox{vol}_{\og} =  \int_{k}^{2k} O(k^{-1} r^{-1-2q}) \, r^2 dr = O(k^{1-2q}) = O(\rho_0^{1-2q}).
\end{align*}
To estimate the integral of the second term on the right hand-side of \eqref{eq:pointwise}, we again use \eqref{eq:gasym2}, \eqref{eq:piasym2} and asymptotic symmetry to derive the first equality.
\begin{align*}
	&\int_{ \rho_0 \le r \le \rho}  \og^{jl} \op_{il} Y^s \overline{\Gamma}_{js}^i \, d\textup{vol}_{\og} = \frac{1}{2} \int_{A_k} \sum_{i,j,l} \tau^k_{ij}Y^l \sigma^k_{ij,l} \, dx + O(\rho_0^{1-2q})\\
	&= - \frac{1}{2} \int_{A_k} \sum_{i,j,l} \sigma^k_{ij}  \tau^k_{ij,l}Y^l \, dx + O(\rho_0^{1-2q}) = - \frac{1}{2} \int_{A_1} \sum_{i,j,l} \sigma_{ij}  \tau_{ij,l}Y^l \, dx + O(\rho_0^{1-2q}),
\end{align*}
where in the second-to-last identity, we integrate by parts and use the fact that $\sigma^k$ and $\tau^k$ vanish on the boundary and that $Y$ is divergence-free.

Combining the above identities, we derive
\[
		8\pi \overline{E} \, \overline{{\bf J}}_p = 8\pi E {\bf J}_p-\int_{A_1} \sum_{i,j,l} \left[\frac{1}{2}\tau_{ij,l} (Y_p)^l +  \tau_{il} (Y_p)^l_{,j}\right] \sigma_{ij}\, d x + O(\rho_0^{1-2q}).
\]
Then we choose $\rho_0$ large so that the error term is less than $\epsilon$. For $k \gg \rho_0$ large enough, we prove that $ (\og^k, \op^k)$ satisfies \eqref{eq:amineq} from Theorem \ref{thm:localpert} applied
with $\vec{\lambda}=-8\pi\vec{\alpha}$.

We show that the center of mass of $(\og, \op)$ remains almost unchanged during the process. For
$p=1,2,3$ we have the components of $\overline{\bf C}$ defined by
\begin{align*}
	\overline{\C}^p&=\frac{1}{16\pi \overline{E} }\lim_{\rho \rightarrow \infty}  \int_{r = \rho} \left[ x^p \sum_{i,j}(\og_{ij,i} - \og_{ii,j}) \frac{x^j}{r} - \sum_i \left( \og_{ip}\frac{x^i}{r} - \og_{ii} \frac{x^p}{r}\right) \right]\, d\sigma_0.
\end{align*}
By the divergence theorem,
\begin{align}
	16\pi \overline{E}\, \overline{\C}^p=&\lim_{\rho \rightarrow \infty} \int_{\rho_0 \le r \le \rho} x^p \sum_{i,j} (\og_{ij,ij} - \og_{ii,jj}) \, dx \notag \\
	&+ \int_{r = \rho_0} \left[ x^p \sum_{i,j}(\og_{ij,i} - \og_{ii,j}) \frac{x^j}{r} - \sum_i \left( \og_{ip}\frac{x^i}{r} - \og_{ii} \frac{x^p}{r}\right) \right]\, d\sigma_0. \label{eq:ctmog}
\end{align}
Because $\sum_{i,j} (\og_{ij,ij} - \og_{ii,jj})$ is the leading order term of the scalar curvature, it can be replaced by the lower order terms such as $|\op|^2$ and $(D\og)^2$ using the constraint equations. Then by \eqref{eq:gasym2}, \eqref{eq:piasym2}, and the symmetry $\sigma^k(x) = \sigma^k(-x)$, the interior integral above is $O(\rho_0^{1-2q})$. Similarly, we have
\begin{align}
	16\pi E \C^p&= \int_{r = \rho_0} \left[ x^p \sum_{i,j}(g_{ij,i} - g_{ii,j}) \frac{x^j}{r} - \sum_i \left( g_{ip}\frac{x^i}{r} - g_{ii} \frac{x^p}{r}\right) \right]\, d\sigma_0 + O(\rho_0^{-1}). \label{eq:ctm}
\end{align}
Moreover, because $\og = (u^k)^4 g$ and $u^k$ is close to $1$ on $\{ r = \rho_0\}$ for $k$ large enough, the difference of the boundary integrals on $\{ r = \rho_0 \}$ is $O(\rho_0^{-1})$. Therefore, for a fixed $\rho_0$ large and for $k\gg \rho_0$ large enough, we have
\[
	|\overline{{\C}} - {\bf C}| \le \epsilon.
\]
\qed
\end{pfam}

%%%%%%%%%%%%%%%%%%%%%%%%%%%%%%%%%%%%%%%%%%%%%%%%%%%%%%%%%%%%%%%%%%%%%%%%%%%%%%
%%%%%%%%%%%%%%%%%%%%%%%%%%%%%%%%%%%%%%%%%%%%%%%%%%%%%%%%%%%%%%%%%%%%%%%%%%%%%%
%%%%%%%%%%%%%%%%%%%                                                                                                                                                   %%%%%%%%%%%%%%%%%%%
%%%%%%%%%%%%%%%%%%%                                                     Center of Mass                                                     %%%%%%%%%%%%%%%%%%%
%%%%%%%%%%%%%%%%%%%                                                                                                                                                   %%%%%%%%%%%%%%%%%%%
%%%%%%%%%%%%%%%%%%%%%%%%%%%%%%%%%%%%%%%%%%%%%%%%%%%%%%%%%%%%%%%%%%%%%%%%%%%%%%
%%%%%%%%%%%%%%%%%%%%%%%%%%%%%%%%%%%%%%%%%%%%%%%%%%%%%%%%%%%%%%%%%%%%%%%%%%%%%%

\section{Specifying the center of mass}
As in the previous section we assume that $g = v^4\delta $ outside a compact set and $\pi=O(r^{-2})$ and
$\pi(x)+\pi(-x)=O(r^{-1-2q})$. We apply Proposition \ref{prop:local} with $\tau=0$ and $\sigma$ chosen by
Theorem \ref{thm:local_cm} to be a solution of $L\sigma=0$ satisfying the  moment condition \eqref{eq:cm}.
\begin{pfctm}
By Proposition \ref{prop:local}, $|\overline{E}^k - E| \le \epsilon$ for $k$ large. Since $\tau=0$,
the proof that the angular momentum satisfies $|\overline{{\bf J}}-{\bf J}|\le\epsilon$ follows as in
the previous section.

To estimate the change in the center of mass we use \eqref{eq:ctmog}, \eqref{eq:ctm}, and  the argument following them. For $p=1,2,3$ we have
\[
	16\pi \left( \overline{E}\, \overline{\C}^p - E \C^p \right)= \lim_{\rho \rightarrow \infty} \int_{ \rho_0\le r \le \rho} x^p \sum_{i,j} (\og_{ij,ij} - \og_{ii,jj}) \, dx + O(\rho_0^{-1}).
\]
Because $\sigma^k(x) \neq \sigma^k(-x)$, the interior term above is not of lower order. Let
$\overline{R}$ denote the scalar curvature of $\og$. Then from the constraint equations, we have
$\overline{R}=O(r^{-4})$ and $\overline{R}(x) - \overline{R}(-x)  = O(r^{-3-2q})$. By \cite[Lemma 3.5]{H10},
\begin{align*}
	&\int_{\rho_0 \le r \le \rho} x^p \sum_{i,j} (\og_{ij,ij} - \og_{ii,jj}) \, dx = \int_{\rho_0 \le r \le \rho} x^p \overline{R} \, dx  - \int_{\rho_0 \le r \le \rho} x^p \mathcal{E}_{\og}\, dx + O( \rho_0^{-1}),
\end{align*}
where
\begin{align*}
	\mathcal{E}_{\og} =&-\sum_{i,j,l} (\og_{il} -\delta_{il}) (2 \og_{ij,lj} - \og_{il,jj} - \og_{jj,li}) \\
	&+\sum_{i,j,l} \big[- \og_{jl,j} \og_{il,i} + \og_{jl,j} \og_{ii,l} + \frac{3}{4} \og_{ij,l} \og_{ij,l} - \frac{1}{4} \og_{jj,l} \og_{ii,l} - \frac{1}{2} \og_{ij,l} \og_{il,j}\big].
\end{align*}
Recall that $\sigma^k \in C_0^{\infty}(A_k)$ is trace-free and divergence-free. By \eqref{eq:gasym2}, the above integral is equal to the following, up to an error term of order $O(\rho_0^{1-2q})$,
\begin{align*}
		&\int_{A_k} x^p \left[ \sum_{i,l} -2 \sigma^k_{il} \left( \frac{2E + A^k}{r} \right)_{,li} - \sum_{i,l,j} \sigma^k_{il} \sigma^k_{il,jj}\right] \, dx \\
		&-\int_{A_k} \frac{3}{4} x^p \sum_{i,j,l} \left[ \left(\frac{ 2E+A^k }{r}\right)_{,l} \delta_{ij} + \sigma^k_{ij,l} \right]^2\, dx \\
		& + \int_{A_k} \frac{1}{2} x^p \sum_{i,j,l}\left[\left( \frac{ 2E+A^k}{r} \right)_{,l} \delta_{ij} + \sigma^k_{ij,l} \right] \left[ \left( \frac{ 2E+A^k}{r} \right)_{,j} \delta_{il} + \sigma^k_{il,j} \right] \, dx\\
%		&+\int_{A_k} x^p\sum_{i,j,l} \left[\sigma^k_{jl,j} \sigma^k_{il,i}- \sigma^k_{jl,j}\sigma^k_{ii,l} - \frac{3}{4} \sigma^k_{ij,l} \sigma^k_{ij,l} + \frac{1}{4} \sigma^k_{jj,l} \sigma^k_{ii,l} + \frac{1}{2} \sigma^k_{ij,l} \sigma^k_{il,j} \right] \, dx\\
		&= \int_{A_k} x^p \sum_{i,l,j}  \left[ - \sigma^k_{il} \sigma^k_{il,jj}- \frac{3}{4} \sigma^k_{ij,l} \sigma^k_{ij,l} \right]\, dx =   \int_{A_1} x^p\sum_{i,j,l} \frac{1}{4} (\sigma_{ij,l})^2\, dx,
\end{align*}
where in the last line we use integration by parts.  We may then apply Theorem \ref{thm:local_cm} with $\vec{\beta}=64\pi E\vec{\gamma}$ to
obtain the required condition (\ref{eq:ctmineq}) on the center of mass.
\qed
\end{pfctm}

%%%%%%%%%%%%%%%%%%%%%%%%%%%%%%%%%%%%%%%%%%%%%%%%%%%%%%%%%%%%%%%%%%%%%%%%%%%%%%
%%%%%%%%%%%%%%%%%%%%%%%%%%%%%%%%%%%%%%%%%%%%%%%%%%%%%%%%%%%%%%%%%%%%%%%%%%%%%%
%%%%%%%%%%%%%%%%%%%                                                                                                                  %%%%%%%%%%%%%%%%%%%
%%%%%%%%%%%%%%%%%%%        Specified the exact angular momentum and center of mass                  %%%%%%%%%%%%%%%%%%%
%%%%%%%%%%%%%%%%%%%                                                                                                                  %%%%%%%%%%%%%%%%%%%
%%%%%%%%%%%%%%%%%%%%%%%%%%%%%%%%%%%%%%%%%%%%%%%%%%%%%%%%%%%%%%%%%%%%%%%%%%%%%%
%%%%%%%%%%%%%%%%%%%%%%%%%%%%%%%%%%%%%%%%%%%%%%%%%%%%%%%%%%%%%%%%%%%%%%%%%%%%%%

\section{Proof of Theorem \ref{thm:specified}}
In Theorem \ref{thm:am} and Theorem  \ref{thm:ctm} we assumed that $g = v^4 \delta$ outside a compact set and $\pi=O(r^{-2})$, $\pi(x)+\pi(-x)=O(r^{-1-2q})$. Using a density theorem \cite{H}, we prove below that the condition can be replaced by the weaker
Regge--Teitelboim condition.
%Recall that a vacuum initial data set $(g,\pi)$ satisfies the Regge--Teitelboim condition if
%\[
%	g(x) = \delta + O(|x|^{-1}), \quad g(x) - g(-x) = O(|x|^{-2}),
%\]
%and
%\[
%	\pi(x) = O(|x|^{-2}), \quad \pi(x) + \pi (-x) = O(|x|^{-3}).
%\]
%Recall that the notation $f = O(r^{-a})$ means that $|f| \le C r^{-a}$ and the first two derivatives of $f$ decay correspondingly faster.
%We remark that our construction works for data $(g,\pi)$ with even weaker assumptions on the decay rates. For simplicity of notation, we assume the decay rates above and leave alone the question of optimal decay conditions.

In this section, we fix the constant $p>3$ and the constant $q\in (1/2,1)$.

\begin{thm} \label{thm:tilde}
Let $(g,\pi)$ be a nontrivial vacuum initial data set satisfying the Regge--Teitelboim condition. Given $\vec{\alpha}, \vec{\gamma} \in \mathbb{R}^3$ and given $\epsilon>0$,  there exists a vacuum initial data set $(\tilde{g}, \tilde{\pi})$ with $\|\tilde{g} -g\|_{W_{-q}^{2,p}} \le \epsilon, \| \tilde{\pi} - \pi \|_{W_{-1-q}^{1,p}} \le \epsilon$, so that
\begin{align} \label{ineq:mp}
		| \tilde{E} - E | \le \epsilon,\quad | \tilde{ \P } - \P | \le \epsilon,
\end{align}
and
\begin{align} \label{ineq:ac}
		 | \tilde{ \J } - \J - \vec{\alpha} | \le \epsilon, \quad  | \tilde{\C} - \C - \vec{\gamma} | \le \epsilon.
\end{align}
\end{thm}
\begin{pf}

By the density theorem in \cite{H}, given the vacuum initial data set $(g,\pi)$ satisfying the Regge--Teitelboim condition and any $\epsilon >0$, there exists a vacuum initial data $(\check{g}, \check{\pi})$ with $\check{g} = v^4 \delta$ outside a compact set and $\check{\pi}(x) = O(r^{-2}), \check{\pi}(x) + \check{\pi}(-x) = O(r^{-1-2q})$, so that $\|\check{g} -g\|_{W_{-q}^{2,p}} \le \epsilon, \| \check{\pi} - \pi \|_{W_{-1-q}^{1,p}} \le \epsilon$. Moreover,
\[
	| \check{E} - E |\le \frac{\epsilon}{3},\;  |\check{\P} - \P | \le \frac{\epsilon}{3}, \; | \check{\J} - \J | \le \frac{\epsilon}{3}, \; |\check{\C} - \C| \le \frac{\epsilon}{3}.
\]
Then we apply the construction in the proof of Theorem \ref{thm:am} to $(\check{g},\check{\pi})$. Given $\vec{\alpha} = (\alpha_1, \alpha_2, \alpha_3) \in \mathbb{R}^3$, there exist symmetric $(0,2)$ tensors $\sigma, \tau \in C^{\infty}_0(A_1)$ satisfying
\begin{align} \label{int:alpha}
	\int_{A_1}  \left[\frac{1}{2} \tau_{ij,l}(Y_p)^l +  \tau_{il} (Y_p)^l_{,j}\right] \sigma^{ij}\, d x = -8\pi \alpha_p,
\end{align}
where rotation vector fields $Y_p = \frac{\partial }{\partial x^p} \times \vec{x}$. Let $\sigma^k (x) = k^{-1}\sigma(x/k)$ and $\tau^k(x) = k^{-2} \sigma(x/k)$. By Proposition \ref{prop:local}, there exists a large integer $k$ so that $(\hat{g}^k, \hat{\pi}^k)$:
\begin{align*}
	&\hat{g}^k = (u^k)^4 (\check{g} + \sigma^k) + h^k,\\
	 &\hat{\pi}^k = (u^k)^2 (\check{\pi} + \tau^k + \mathcal{L}_{(g+ \sigma^k)} {\bf X}^k) + w^k
\end{align*}
satisfies the vacuum constraint equations, where $(u^k, {\bf X}^k)$ and $(h^k, w^k)$ arise from solving the linearized constraint equations. They satisfy the decay condition \eqref{eq:decay}, and  the $(h^k, w^k)$ have compact support. Then, from Theorem \ref{thm:am}, we have
\[
		|\hat{E} - \check{E}| \le \frac{\epsilon}{3},\;\quad | \hat{\P} - \check{\P}| \le \frac{\epsilon}{3}, \; \quad| \hat{\C} - \check{\C} | \le \frac{\epsilon}{3},
\]
and
\[
		 | \hat{\J} - \check{\J} - \vec{\alpha}| \le \frac{\epsilon}{3}.
\]
We suppress the superscript $k$ of $(\hat{g}^k, \hat{\pi}^k)$ in the following. Clearly $(\hat{g}, \hat{\pi} )$ satisfies the condition in Theorem \ref{thm:ctm}, namely $\hat{g} = (uv)^4 \delta$ outside a compact set and $\hat{\pi} = O(r^{-2}), \hat{\pi}(x) +  \hat{\pi}(-x) = O(r^{-1-2q})$. Let $\hat{\sigma}$ be a symmetric $(0,2)$ tensor satisfying
\begin{align} \label{int:gamma}
			\int_{A_1} x^p \sum_{i,j,k}(\hat{\sigma}_{ij,k})^2\, d x = 64\pi \hat{E} \gamma_p.
\end{align}
Let $\hat{\sigma}^l (x) = l^{-1} \hat{\sigma}(x/l)$. By the construction in the proof of Theorem \ref{thm:ctm}, there exists an integer $l \gg k$ so that
\begin{align*}
	\tilde{g}^l &= (\hat{u}^l)^4(\hat{g} + \hat{\sigma}^{l}) + \hat{h}^l,\\
	\tilde{\pi}^l &= (\hat{u}^l)^2 (\hat{\pi} + \mathcal{L}_{(\hat{g} + \hat{\sigma}^{l})} \hat{\bf X}^l ) + \hat{w}^l
\end{align*}
satisfies the vacuum constraint equations, where $(\hat{u}^l, \hat{{\bf X}}^l)$ and $(\hat{h}^l, \hat{w}^l)$ are from solving the linearized constraint equations as above. Moreover,
\[
		|\tilde{E} - \hat{E}| \le \frac{\epsilon}{3},\quad | \tilde{\P} - \hat{\P} | \le \frac{\epsilon}{3}, \quad|\tilde{\J} - \hat{\J}| \le \frac{\epsilon}{3},
\]
and
\[
	 | \tilde{\C} - \hat{\C} - \vec{\gamma} | \le \frac{\epsilon}{3}.
\]
Then \eqref{ineq:mp} and \eqref{ineq:ac} follow by combining the above inequalities.
\qed
\end{pf}

We can further perturb $(\tilde{g}, \tilde{\pi})$ so that the energy-momentum vector equals to that of $(g,\pi)$, while changing the angular momentun and center of mass by only a small amount.

\begin{prop} \label{prop:same_mp}
Let $(g,\pi)$ be a nontrivial vacuum initial data set satisfying the Regge--Teitelboim condition. Given $\vec{\alpha}, \vec{\gamma} \in \mathbb{R}^3$ and $\epsilon >0$,  there exists a vacuum initial data $(\bg, \bp)$ satisfying  $\|\bg -g\|_{W_{-q}^{2,p}} \le \epsilon$ and  $\| \bp - \pi \|_{W_{-1-q}^{1,p}} \le \epsilon$ so that $\overline{E} = E$ and $\overline{ \P } = \P$ and
\[
		 | \overline{ \J } - \J - \vec{\alpha} | \le \epsilon, \quad \mbox{and} \quad  | \overline{\C} - \C - \vec{\gamma} | \le \epsilon.
\]
\end{prop}
\begin{pf}
Let $(\tilde{g}, \tilde{\pi})$ be the initial data constructed in Theorem \ref{thm:tilde}. Let $(\bar{g}, \bar{\pi})$ be the vacuum initial data from scaling $\bg = \lambda^2 \tilde{g}$ and $\bp = \lambda \tilde{\pi}$, where the constant $\lambda$ is a positive constant and $\lambda^2 =(E^2 - | \P |^2)/ (\tilde{E}^2 - | \tilde{\P}|^2)$. Then, by straightforward computations, we have $(\overline{E}, \overline{\P}, \overline{\J}, \overline{\C})= \lambda (\tilde{E}, \tilde{\P}, \tilde{\J}, \tilde{\C})$ and then
\[
		\overline{E}^2 - | \overline{ \P }|^2 = E^2 - | \P|^2.
\]
Notice that by \eqref{ineq:mp},
\[
		 \Big| (\tilde{E}^2 - |\tilde{\P}|^2) -  (E^2 - | \P|^2 )  \Big|\le  2\epsilon (E + |\P|) + \epsilon^2.
\]
Therefore, since $E > |\P|$ by the positive mass theorem,  we divide the above inequality by $E^2 - | \P|^2$. Then
\[
	|\lambda^{-2}-1| \le \frac{2\epsilon }{E -| \P| } + \frac{\epsilon^2}{E^2 - | \P |^2}.
\]
Because $E$ and $\P$ are fixed, we can choose $k,l$ in the proof of Theorem \ref{thm:tilde} large enough so that $\lambda$ is close to $1$. Therefore, $(\overline{E}, \overline{\P}, \overline{\J}, \overline{\C})$ is close to $(\tilde{E}, \tilde{\P}, \tilde{\J}, \tilde{\C})$ and hence to $(E,\P,\J,\C)$. %Also, we have
%\[
%		\overline{E}^2 - | \overline{ \P }|^2 = E^2 - | \P|^2.
%\]
Because $\overline{E}$ is close to $E$, we then boost the data $(\bg, \bp)$ by a small angle so that $\overline{E} = E$, and then $|\overline{\P}| = | \P|$ (the existence of such boosted slice is proven in \cite{CO}). By rotating the asymptotically flat coordinates, we can make $\overline{\P} = \P$. Also, notice that the angular momentum and center of mass only change a small amount after these transformations.  (The transformation formulas of these quantities under the Poincar\'{e} transformations can be found in, for example, \cite[Appendix E]{CD}.)
\qed
\end{pf}

To prove Theorem \ref{thm:specified}, we need the following degree argument.
\begin{lemma} \label{lemma:degree}
Fix the constant $a>0$. Let $B_{a}(z_0) \subset \mathbb{R}^n$ denote the closed ball centered at $z_0$ with radius $a$. Let $f: B_{a}(z_0) \rightarrow \mathbb{R}^n$ be a continuous map satisfying, for any $z \in B_{a}(z_0)$,
\[
	|f(z) - z | \le a.
\]
Then $f^{-1}(z_0)$ is non-empty. More precisely,  either $f(z) = z_0$ for some $z\in \partial B_a(z_0)$ or the degree of $f$ at $z_0$ is one.
\end{lemma}
\begin{pf}
By scaling, we only need to prove the case when $a=1$. We define the continuous homotopy between $f$ and the identity map for $0 \le t \le 1$:
\[
	h(z,t) = (1-t) z + t f(z) .
\]
For a boundary point $z\in \partial B_1 (z_0)$,
\begin{align*}
		(h(z,t) - z_0 )\cdot (z-z_0) &= \left[ (z-z_0) + t (f(z) - z) \right] \cdot (z -z_0) \\
		&\ge |z-z_0|^2 - t | f(z) - z| |z-z_0| \ge 1-t.
\end{align*}
Then either $f(z) = z_0$ for some $z\in \partial B_1(z_0)$ or $h(z,t) \ne z_0 $ for all $0\le t \le 1$ and for all $z \in {\partial B_1(z_0)} $. In particular, the latter case implies that $z_0$ stays in the range of $h(\cdot, t)$ for all $t\in [0,1]$. Therefore, $f^{-1}(z_0)$ is non-empty.
\qed
\end{pf}

\begin{pfsp}
Denote the given constant vector $(\vec{\alpha}_0, \vec{ \gamma}_0)$ by $z_0 \in \mathbb{R}^6$. We may without loss of generality prove only for the case  $|\vec{\alpha}_0|\neq 0$ and $|\vec{\gamma}_0| \neq 0$, for if $\vec{\alpha}_0$ (or $\vec{\gamma}_0$) is the zero vector, we apply the theorem twice to a non-zero constant vector $\vec{v}$ and then to $-\vec{v}$. We define the map $f: B_{\epsilon} (z_0) \subset \mathbb{R}^6 \rightarrow \mathbb{R}^6$ by
\[
	f(\vec{\alpha}, \vec{\gamma}) = (\overline{ \J} - {\bf J}, \overline{\C} - {\bf C}),
\]
where $\overline{\J}$ and $\overline{\C}$ are the angular momentum and center of mass of $(\bg, \bp)$ constructed in Proposition \ref{prop:same_mp}. By the construction,
\[
	|f(z) - z| \le \epsilon.
\]
Once we verify that $f$ is continuous, we apply Lemma \ref{lemma:degree} to obtain $f(\vec{\alpha}, \vec{\gamma}) = (\vec{\alpha}_0, \vec{\gamma}_0)$ for some $(\vec{\alpha}, \vec{\gamma})$ and  complete the proof of Theorem \ref{thm:specified}.
%If $f(z)=z_0$ for some $z\in  \partial B_{a} (z_0)$, we are done. If not, then we define the continuous homotopy between $f$ and the identity map for $t\in [0,1]$:
%\[
%	h(z,t) = (1-t) z + t f(z) .
%\]
%For $z\in \partial B_{a} (z_0)$,
%\begin{align*}
%		(h(z,t) - z_0 )\cdot (z-z_0) &= \left[ (z-z_0) + t (f(z) - z) \right] \cdot (z -z_0) \\
%		&\ge |z-z_0|^2 - t | f(z) - z| |z-z_0| \ge a( a - t  \epsilon).
%\end{align*}
%Therefore, because $a>\epsilon$, $|h(z,t) - z_0 |>0$ for all $z \in {\partial B_{a}(z_0)} $ and $t\in[0,1]$. That is, the homotopy keeps $z_0$ in the range of $h(\cdot, t)$ for all $0 \le t \le 1$, and the degree of $f$ at $z_0$ is $1$. This implies that $f^{-1}(z_0) \neq \phi$. In either case, there is some $z\in  \overline{B_{a} (z_0)}$ so that $f(z) = z_0$.

\begin{claim}
The map $f$ is continuous.
\end{claim}
\begin{pf}
Given non-zero vectors $(\vec{\alpha}_0, \vec{\gamma}_0)$, let $(\sigma, \tau,  \hat{\sigma})$ be the symmetric $(0,2)$ tensors and $k,l$ be the integers  in the proof of Theorem \ref{thm:tilde} for $(\vec{\alpha}_0, \vec{\gamma}_0)$. (Notice that $k,l$ may be chosen large depending only on $\vec{\alpha}_0, \vec{\gamma}_0$, and $\epsilon$.)

Assume that $(\vec{\alpha}, \vec{\gamma})$ is another pair of constant vectors. We fix the symmetric $(0,2)$ tensors $\sigma, \tau, \hat{\sigma}$ and the annular shells determined by $k$ and $l$. We choose the symmetric $(0,2)$ tensors for $\vec{\alpha}, \vec{\gamma}$ from $\sigma, \tau, \hat{\sigma}$ and apply the construction in Theorem \ref{thm:tilde} over the annular shells determined by $k$ and $l$.

That the choice of the symmetric $(0,2)$ tensors depends continuously on $\vec{\alpha}$ can be seen as follows: %The choice of the symmetric $(0,2)$ tensors for $\vec{\alpha}$ is continuous as follows:
Let $R_1: \mathbb{R}^3 \rightarrow \mathbb{R}^3$ be a rotation so that $R_1 (\vec{\alpha}_0)$ is parallel to $\vec{\alpha} $. The symmetric $(0,2)$ tensors defined by
\[
	\sqrt{\frac{| \vec{\alpha}|}{ | \vec{\alpha}_0 |}}(R_1^{-1})^*\sigma,\quad \sqrt{\frac{| \vec{\alpha}|}{ | \vec{\alpha}_0 |}} (R_1^{-1})^* \tau
\]
 satisfy the corresponding condition \eqref{int:alpha} for $\vec{\alpha}$ on the right-hand side. It is easy to check that at each point
 \[
 	\bigg|\sqrt{\frac{| \vec{\alpha}|}{ | \vec{\alpha}_0 |}}(R_1^{-1})^*\sigma - \sigma \bigg| \le | \vec{\alpha} - \vec{\alpha}_0| \left( \frac{1}{2 | \vec{\alpha}_0 | } |\sigma|+ \frac{3}{2} | D\sigma| \right).
 \]
%Let $R_1, R_2: \mathbb{R}^3 \rightarrow \mathbb{R}^3$ be rotations so that $R_1 (\vec{\alpha}_0) $ and $R_2(\vec{\gamma}_0)$ are respectively parallel to $\vec{\alpha} $ and $\vec{\gamma}$. The symmetric $(0,2)$ tensors defined by
%\[
%	\sqrt{\frac{| \vec{\alpha}|}{ | \vec{\alpha}_0 |}}(R_1^{-1})^*\sigma,\quad \sqrt{\frac{| \vec{\alpha}|}{ | \vec{\alpha}_0 |}} (T_1^{-1})^* \tau, \quad \sqrt{\frac{| \vec{\gamma}|}{ | \vec{\gamma}_0 |}} (T_2^{-1})^* \hat{\sigma}
%\]
% satisfy the corresponding integrals \eqref{int:alpha} and \eqref{int:gamma} for $\vec{\alpha}, \vec{\gamma}$ on the right hand sides. It is easy to check that, pointwisely,
 %\[
 %	\bigg|\sqrt{\frac{| \vec{\alpha}|}{ | \vec{\alpha}_0 |}}(R_1^{-1})^*\sigma - \sigma \bigg| \le | \vec{\alpha} - \vec{\alpha}_0| \left( \frac{1}{2 | \vec{\alpha}_0 | } |\sigma|+ \frac{3}{2} | D\sigma| \right) %\bigg| \sqrt{\frac{| \vec{\alpha}|}{ | \vec{\alpha}_0 |}} - 1 \bigg| |\sigma| + \left[ ( \big| | \vec{\alpha}|^{1/2} - |\vec{\alpha_0}|^{1/2} \big| ) |\vec{\alpha}_0|^{1/2}+ | \vec{\alpha}_0 - \vec{\alpha}| \right] |D \sigma|.
 %\]
Similar estimate can be derived for the other tensor $\tau$. For $\vec{\gamma}$ and the integral \eqref{int:gamma}, we can also choose the tensor corresponding to $\hat{\sigma}$ in the same fashion. It is straightforward to check that the rest of the construction is continuous, and hence $f$ is continuous.

%It is easy to see that $T_1$ rotates and scales $\vec{\alpha}_0$ to $\vec{\alpha}$ and  $T_2$ rotates and scales $\vec{\gamma}_0$ to $\vec{\gamma}$. Denote the identity map from $\mathbb{R}^3$ to $\mathbb{R}^3$  by $I$. Then $|T_1 - I | = \frac{ |\vec{\alpha} - \vec{\alpha}_0|}{ |\vec{\alpha}_0|} $ and $|T_2 - I| = \frac{ |\vec{\gamma}- \vec{\gamma}_0|}{ |\vec{\gamma}_0| }$. We then make the choice of the perturbations $\sigma, \tau, \hat{\sigma}$ continuously in $(\vec{\alpha}, \vec{\gamma})$. It is straightforward to check that the rest of construction is continuous, and hence $f$ is continuous.

%there exists $(\bg_{(\vec{\alpha}_0, \vec{\gamma}_0)}, \bp_{(\vec{\alpha}_0, \vec{\gamma}_0)})$ satisfying
%\[
%	\| \bg_{(\vec{\alpha}_0,\vec{\gamma}_0)}  - g\|_{W^{2,p}_{-q}} \le \epsilon, \quad \| \bp_{(\vec{\alpha}_0, \vec{\gamma}_0)} - \pi \|_{W^{2,p}_{-1-q}},
%\]
%with $\overline{E} = E$ and $\overline{ \P } = \P$ and
%\[
%		 | \overline{ \J } - \J - \vec{\alpha}_0 | \le \epsilon, \quad \mbox{and} \quad  | \overline{\C} - \C - \vec{\gamma}_0 | \le \epsilon.
%\]
\qed
\end{pf}
\end{pfsp}

\begin{cor}
Given any constant vector $(E,\P,\J,\C) \in \mathbb{R}^{10}$ with $E > | \P|$, there exists a smooth and complete asymptotically flat vacuum initial data set whose energy, linear momentum, angular momentum, and center of mass are the corresponding components of this constant vector.
\end{cor}
\begin{pf}
By Theorem \ref{thm:specified}, it suffices to show that there exists a vacuum initial data with the specified $E$ and $\P$. By the results of the global existence of the Cauchy problem (see \cite{CK, KN, LR}), given a strongly asymptotically flat vacuum initial data set $(g,\pi)$ close to the flat data, there exist future and past complete vacuum developments. In particular, we can boost the slice in spacetime and then rotate the coordinates so that the energy-momentum vector of $(g,\pi)$ is parallel to the given vector $(E,\P)$. Then, by scaling the data, we obtain a vacuum initial data with the desired energy-momentum vector $(E,\P)$.
\qed
\end{pf}

\bibliographystyle{plain}

\begin{thebibliography}{HKW}
%\bibitem{Bartnik87} Bartnik, R., \emph{The mass of an asymptotically flat manifold}, Comm. Pure Appl. Math. 39 (1986), no. 5, 661-693.

\bibitem{BO} Beig, R. and \'{O} Murchadha, N. \emph{The Poincar\'{e} group as the symmetry group of canonical general relativity}, Ann. Physics 174 (1987), no. 2, 463--498.
\bibitem{CK} Christodoulou, D. and Klainerman, S., \emph{The global nonlinear stability of the Minkowski space}, Princeton Mathematical Series, 41. Princeton University Press.
\bibitem{CO} Christodoulou, D. and \'{O} Murchadha, N., \emph{The boost problem in general relativity}, Comm. Math. Phys. 80 (1981), no. 2, 271--300.


\bibitem{Ch1} Chru\'{s}ciel, P. T., {\it Mass and angular-momentum inequalities
 for axi-symmetric initial data sets. I. Positivity of mass}, Ann. Physics 323 (2008), no. 10, 2566--2590.

\bibitem{CCI} Chru\'{s}ciel, P. T., Corvino, J., and Isenberg, J., \emph{Construction of N-body initial data sets in general relativity}, arXiv:1004.1355.
\bibitem{CD} Chru\'{s}ciel, P. T. and Delay, E., \emph{On mapping properties of the general relativistic constraints operator in weighted function spaces, with applications},  M\'{e}m. Soc. Math. Fr. (N.S.) No. 94 (2003).

\bibitem{Ch2} Chru\'{s}ciel, P. T., Li, Y., and Weinstein, G., {\it Mass and angular-momentum inequalities
 for axi-symmetric initial data sets. II. Angular momentum}, Ann. Physics 323 (2008), no. 10, 2591--2613.


\bibitem{Ch3} Chru\'{s}ciel, P. T. and Costa, J., {\it Mass angular-momentum and charge inequalities
 for axisymmetric initial data}, Class. Quantum Grav. 26 (2009), no. 23, 235013, 7 pp.


\bibitem{C} Corvino, J., {\it Scalar curvature deformation and a gluing construction for the Einstein constraint equations}, Comm. Math. Phys. 214 (2000), no. 1, 137--189.
\bibitem{CS06} Corvino, J. and Schoen, R., {\it On the asymptotics for the vacuum Einstein constraint equations}, J. Diff. Geom. Volume 73, Number 2 (2006), 185--217.

\bibitem{D} Dain, S., {\it Proof of the angular momentum-mass inequality for axisymmetric black holes},
J. Diff. Geom. 79 (2008), 3--67.

\bibitem{H} Huang, L.-H., {\it On the center of mass of isolated physical systems with general asymptotics}, Class. Quantum Grav. 26 (2009), no. 1, 015012, 25 pp.

\bibitem{H10} Huang, L.-H., {\it Solutions of special asymptotics to the Einstein constraint equations}, Class. Quantum Grav. 27 (2010), no. 24,  245002, 10 pp.


\bibitem{KN} Klainerman, S. and Nicol\`{o}, F., \emph{The evolution problem in general relativity}, Progress in Mathematical Physics, 25.
\bibitem{LR} Lindblad, H. and Rodnianski, I., \emph{Global existence for the Einstein vacuum equations in wave coordinates}, Comm. Math. Phys. 256 (2005), no. 1, 43--110.

\bibitem{RT} Regge, T. and Teitelboim, C., {\it Role of Surface Integrals in the Hamiltonian Formulation
of General Relativity}, Ann. Phys. Volume 88 (1974), 286--318.

\bibitem{Z} Zhang, X., {\it Angular momentum and positive mass theorem}, Comm. Math. Phys. 206 (1999), 137--155.




\end{thebibliography}

\end{document}